\documentclass[11pt]{article} 
\usepackage{amsmath,amsthm} 
\usepackage{amssymb,latexsym}
\usepackage[T1]{fontenc}
\usepackage{authblk}
\usepackage{lipsum}
\usepackage{hyperref}
\usepackage{lipsum} 
\usepackage{lineno}
\usepackage{upgreek}
\usepackage{tikz}
\usetikzlibrary{shapes.geometric}
\usetikzlibrary{positioning}
\usetikzlibrary{arrows.meta}
\modulolinenumbers[1]
\begin{document}
	\newtheorem{theorem}{Theorem}[section]
	\newtheorem{question}{Question}[section]
	\newtheorem{thm}[theorem]{Theorem}
	\newtheorem{lem}[theorem]{Lemma}
	\newtheorem{eg}[theorem]{Example}
	\newtheorem{prop}[theorem]{Proposition}
	\newtheorem{cor}[theorem]{Corollary}
	\newtheorem{rem}[theorem]{Remark}
	\newtheorem{deff}[theorem]{Definition}
	\numberwithin{equation}{section}
	
	\title{Cohen, Levitzki, Hilbert Basis, and Lasker-Noether Theorems for Nil-$S$-Noetherian Rings}
	
	\author[1]{Aman Pandey}
	\author[2]{Ajim Uddin Ansari}
	\author[3]{Tushar Singh}

	\affil[1, 2]{\small Department of Mathematics, CMP Degree College, University of Allahabad, Prayagraj-211002, India \vskip0.01in Emails: amanpandey1931@gmail.com, 
		ajimmatau@gmail.com}
	\vskip0.05in
	\affil[3]{\small Department of Mathematics, Motilal Nehru National Institute of Technology Allahabad, Prayagraj 211004, India \vskip0.01in Emails: sjstusharsingh0019@gmail.com, tushar.2021rma11@mnnit.ac.in }
	\date{}
\maketitle
	\hrule
	\maketitle
	\hrule
	
\begin{abstract}
\noindent

In this paper, we introduce a new class of rings called Nil-$S$-Noetherian rings, which generalizes both $S$-Noetherian rings and $Nil_{*}$-Noetherian rings. We investigate several properties of this new class and establish generalized versions of some classical results, including Cohen's theorem, Levitzki's theorem, and Hilbert's basis theorem. Furthermore, we prove $S$-version of  classical Lasker-Noether theorem for Nil-$S$-Noetherian rings.
\end{abstract}

	\hrule
	
\section{Introduction}
The theory of Noetherian rings has been playing an important role in the development of structure theory of commutative rings. One of the roots of this theory is the historical article \cite{ne21} by Noether in 1921. Recall that a ring is called  Noetherian if it satisfies ascending chain condition on ideals. In the past few decades, many mathematicians have investigated how the Noetherian property can be applied to various classes of rings and have attempted to generalize the concept of Noetherian rings because of its importance (see \cite{ah18}, \cite{ad02}, \cite{badawi03}, \cite{jw16},  \cite{zb17}, \cite{nonnil20}, \cite{ts23}, \cite{ts24}, and \cite{ts25}).  Nil$_{*}$-Noetherian rings,  Nonnil-Noetherian rings, $S$-Noetherian rings, and Nonnil-$S$-Noetherian rings are typical generalizations of Noetherian rings.

  In 2002, Anderson and Dumitrescu \cite{ad02} introduced the concept of $S$-Noetherian rings. Let $R$ be a commutative ring and $S$ be a multiplicatively closed subset of $R$. An ideal $I$ of $R$ is said to be $S$-finite if there exists $s \in S$ such that $sI \subseteq J \subseteq I$ for some finitely generated ideal $J$ of $R$. A ring $R$ is called $S$-Noetherian if every ideal of $R$ is $S$-finite. Anderson and Dumitrescu extended several classical results on Noetherian rings to $S$-Noetherian rings, including Cohen's theorem and Hilbert basis theorem.
  In 2003, Badawi \cite{badawi03} introduced the concept of nonnil-Noetherian rings. An ideal $I$ of $R$ is called nonnil if $I \nsubseteq Nil(R)$, where $Nil(R)$ denotes the nilradical of $R$. A ring $R$ is called nonnil-Noetherian if every nonnil ideal of $R$ is finitely generated. 
   In 2020, Kwon and Lim \cite{nonnil20} introduced the concept of nonnil-$S$-Noetherian rings as a generalization of both $S$-Noetherian rings and nonnil-Noetherian rings. A ring $R$ is called nonnil-$S$-Noetherian if every nonnil ideal of $R$ is $S$-finite.
  Motivated by these notions, Xiaolei \cite{zx22} introduced another generalization of Noetherian rings called $Nil_{*}$-Noetherian rings. An ideal $I$ of $R$ is called a nil ideal if $I \subseteq Nil(R)$. A ring $R$ is called a $Nil_{*}$-Noetherian ring if every nil ideal of $R$ is finitely generated. It is clear that a ring is Noetherian if and only if it is both nonnil-Noetherian and $Nil_{*}$-Noetherian. Xiaolei also generalized several classical results of Noetherian rings to $Nil_{*}$-Noetherian rings, including the Cartan--Eilenberg--Bass theorem and Hilbert basis theorem.
  
   The main  purpose of this paper is to introduce and study  $S$-version of Nil$_{*}$-Noetherian rings, which we  call Nil-$S$-Noetherian ring.
   	We say that a ring $R$ is  Nil-$S$-Noetherian if each nil ideal of $R$ is $S$-finite.
    It is clear that every Noetherian ring is $S$-Noetherian, every $Nil_{*}$-Noetherian ring is Nil-$S$-Noetherian, and every $S$-Noetherian ring is Nil-$S$-Noetherian. The following diagram illustrates these implications. However, the converse of each implication does not hold in general (see Examples~\ref{n12} and~\ref{n13}).
\begin{center}

	\begin{tikzpicture}[
	node distance = 2cm and 3cm,
	every node/.style = {rectangle, draw, minimum height=0.8cm, minimum width=2.5cm, align=center},
	arrow/.style = {-{Stealth[scale=1.2]}, thick}
	]
	
	\node (noetherian)       {Noetherian};
	\node (nilnoetherian)    [right=of noetherian] {Nil$_{*}$-Noetherian};
	\node (snoetherian)      [below=of noetherian] {$S$-Noetherian};
	\node (snilnoetherian)   [below=of nilnoetherian] {Nil-$S$-Noetherian};
	
	\draw[arrow] (noetherian) -- (nilnoetherian);        
	\draw[arrow] (noetherian) -- (snoetherian);
	\draw[arrow] (nilnoetherian) -- (snilnoetherian);
	\draw[arrow] (snoetherian) -- (snilnoetherian);

	\end{tikzpicture}
\end{center}
Among other results, we prove Cohen's theorem, Levitzki's theorem, and Hilbert basis theorem for Nil-$S$-Noetherian rings (see Theorems \ref{cohen}, \ref{Levitzki-1}, \ref{Levitzki-2}, \ref{Hilbert-1}, and \ref{Hilbert-2}). 
The classical Lasker--Noether theorem states that in a commutative Noetherian ring every ideal can be expressed as the intersection of finitely many primary ideals, called a primary decomposition \cite{le05, ne21}. 
Recently, Singh et al. ~\cite{ts24} extended this classical theorem by establishing the existence of $S$-primary decomposition in $S$-Noetherian modules, thereby demonstrating that the Lasker-Noether theorem remains valid in this generalized setting. Motivated by these developments, in this paper we further generalize the theory by proving the Lasker-Noether theorem for Nil-$S$-Noetherian rings (see Theorem \ref{on}).

  Throughout this paper, all rings are assumed to be commutative with identity. We denote by $\mathrm{Nil}(R)$ the nilradical of the ring  $R$, that is, the set of all nilpotent elements of $R$, and by $Z(R)$ the set of all zero divisors of $R$, unless otherwise stated.
  
\section{Nil-$S$-Noetherian Rings}
We begin the section by introducing the notion of Nil-$S$-Noetherian rings, which unifies and extends both \( \text{Nil}_* \)-Noetherian rings and $S$-Noetherian rings. 

\begin{deff}\label{n1}
Let $R$ be a  ring and $S \subseteq R $ be a multiplicatively closed subset.	We say that $R$ is  Nil-$S$-Noetherian if each nil ideal of $R$ is $S$-finite.
\end{deff}

Clearly every Nil$_*$-Noetherian ring and every $S$-Noetherian ring is a Nil-$S$-Noetherian ring; however, the converse does not hold in general. The following example shows that a Nil-$S$-Noetherian ring need not be Nil$_{*}$-Noetherian.
\begin{eg}\label{n12}
	\noindent
	Let $R = \prod\limits_{n \geq 1} \mathbb{Z}_{p^n}$, where $p$ is a prime, and let $S=\{ 1_{R}, s \}$, where $s= (\bar{1}, \bar{1}, 0, 0,  \ldots)$.
	For each $i\geq 2$, let $a_i \in R$ be the element of $R$ whose $i$-th component is $\overline{p^{\,i-1}}$ and all other components are zero, that is,  $a_i=(0,\ldots,0,\overline{p^{\,i-1}},0,\ldots)$, where $\overline{p^{\,i-1}}$ appears in the $i$-th component.
	 Then  $I=\langle a_2,a_3, a_4,  \ldots \rangle$ is an infinitely generated ideal with $I^2 = (0)$.   Thus $I$ is an infinitely generated nil ideal of $R$, and hence $R$ is not  Nil$_{*}$-Noetherian.

	Let $J$ be an ideal of $R$. If $J$ is of the form $\{(0,0, \alpha_{1}, \alpha_{2},\alpha_{3},\ldots)\mid \alpha_{i}\in\mathbb{Z}_{p^{i+2}}~ for~i=1,2,3,\ldots\}$. Then for $s=(\bar{1},\bar{1},0,0,\ldots)\in S$, we have  $sJ=0$. Thus $sJ \subseteq F \subseteq J$, where $F=(0)$. 
	Further, if $J$ is of the forms  $(\alpha,0,0,\ldots)$ or $(0,\beta,0,\ldots)$, then again $sJ\subseteq F\subseteq J$ for  finitely generated ideal  $ F=\mathbb{Z}_p \times \mathbb{Z}_{p^2} \times 0 \times 0 \times \cdots$ of $R$.
	In all other cases,  we obtain 
	$sJ \subseteq F \subseteq J$
	where  $F=\mathbb{Z}_p \times \mathbb{Z}_{p^2} \times 0 \times 0 \times \cdots$
	is finitely generated. Hence every ideal of $R$ is $S$-finite. Thus $R$ is an $S$-Noetherian ring, and consequently $R$ is a Nil-$S$-Noetherian ring.
	\end{eg}

The following example shows that a Nil-$S$-Noetherian ring need not be $S$-Noetherian.

\begin{eg}\label{n13}
	Let $R = k[x_1, x_2, x_3,\ldots]$, and $S=\{x_1^n \mid n\in \mathbb{Z}_{\ge 0}\}$. Consider the ideal
	$I=\langle x_2,x_3,x_4,\ldots \rangle$. We claim that there exists no $s\in S$ and no finitely generated ideal $F$ of $R$ such that $sI \subseteq F \subseteq I$.
	Suppose, to the contrary, that such an element $s\in S$ and a finitely generated ideal $F$ exist. Then $s=x_1^m$ for some $m\ge 0$. Since $F\subseteq I$ and $F$ is finitely generated, there exist indices $2\le i_1,i_2,\ldots,i_t$ such that $F=\langle  x_{i_1},x_{i_2},\ldots,x_{i_t} \rangle$. Now, $sI = x_1^m  \langle x_2,x_3,x_4,\ldots \rangle= \langle x_1^m x_2,\, x_1^m x_3,\, x_1^m x_4,\, \ldots \rangle$. Consequently, $sI\nsubseteq F$ since $x_{1}\notin F$. This contradicts the assumption that $sI\subseteq F$. Therefore, no such $s\in S$ and finitely generated ideal $F$ exist. Hence, $I$ is not $S$-finite, and so $R$ is not an $S$--Noetherian ring. Moreover, since $R$ is an integral domain, we have $\operatorname{Nil}(R)=0$. Consequently, every nil ideal of $R$ is trivially $S$-finite, and hence $R$ is a Nil-$S$--Noetherian ring.
\end{eg}

\begin{prop}
	A ring $R$ is $S$-Noetherian if and only if it is both Nil-$S$-Noetherian and Nonnil-$S$-Noetherian.
\end{prop}

\begin{proof}
	Suppose that $R$ is an $S$-Noetherian ring. Then every ideal of $R$ is $S$-finite. In particular, every nil ideal and every nonnil ideal of $R$ is $S$-finite. Hence $R$ is both Nil-$S$-Noetherian and Nonnil-$S$-Noetherian.
	Conversely, assume that $R$ is both Nil-$S$-Noetherian and Nonnil-$S$-Noetherian. Let $I$ be an ideal of $R$. If $I \subseteq \mathrm{Nil}(R)$, then $I$ is $S$-finite. If $I \nsubseteq \mathrm{Nil}(R)$, then $I$ is $S$-finite. Thus, in either case, $I$ is $S$-finite. Therefore, every ideal of $R$ is $S$-finite, and hence $R$ is $S$-Noetherian.
\end{proof}

Recall \cite{zb17}, let $E$ be a family of ideals of a ring $R$. An element $I\in E$ is said to be an \textit{$S$-maximal element} of $E$ if there exists an $s\in S$ such that for each $J\in E$, if $I\subseteq J$, then $sJ\subseteq I$. Also, a chain of ideals $(I_{i})_{i\in\wedge}$  of $R$ is called \textit{$S$-stationary} if there exist $k\in \wedge$ and $s\in S$ such that  $sI_{i}\subseteq I_{k}$ for all $i\in\wedge$, where $\wedge$ is an arbitrary indexing set.  A family $\mathcal{F}$ of ideals of $R$ is said to be $S$-saturated if it satisfies the following property: for every ideal $I$ of $R$, if there exist $s\in S$ and $J\in\mathcal{F}$ such that $sI\subseteq J$, then $I\in\mathcal{F}$.

\begin{theorem}\label{j1}
	Let $R$ be  a  ring. Then the following statements are equivalent.
	\begin{enumerate}
		\item $R$ is  Nil-$S$-Noetherian.
		
		\item  Every ascending chain of nil ideals of $R$ is $S$-stationary.
		
		\item Every nonempty $S$-saturated set of  nil ideals of $R$ has a maximal element.
		
		\item Every nonempty family of nil ideals has an $S$-maximal element with respect to inclusion.
	\end{enumerate}
\end{theorem}

\begin{proof}
	\leavevmode
	
	\item(1) $\Rightarrow$ (2). Let $(I_n)_{n\in\wedge}$ be an increasing sequence of nil ideals of $R$. Define the ideal $I = \bigcup\limits_{n\in\wedge}I_n$. Let $\alpha\in I$. Then $\alpha\in I_{k}$ for some $k\in\wedge$. Since $I_{k}$ is nil, there exists $m\in\mathbb{N}$ such that $\alpha^{m}=0$, and hence $I$ is nil. This implies that $I$ is $S$-finite since $R$ is Nil-$S$-Noetherian. Consequently, there exist a finitely  generated ideal $F\subseteq R$ and $s\in S$ such that $sI\subseteq F\subseteq I$. Since $F$ is finitely generated, there is a $k\in\wedge$ satisfying $F\subseteq I_{k}$. Then we have $sI\subseteq F\subseteq I_{k}$, from which it follows that $sI_n \subseteq I_{k}$ for each $n\in\wedge$.\\
	
	\item (2) $\Rightarrow$ (3).  Let $\mathcal{D}$ be an $S$-saturated set of nil ideals of $R$. Given any chain \(\{I_n\}_{n \in\wedge} \subseteq \mathcal{D} \), we claim that \( I = \bigcup\limits_{n\in\wedge} I_n \) belongs to $\mathcal{D}$, which will establish that \( I \) as an upper bound for the chain. Since each $I_n$ is a nil ideal of $R$, so $I$ is a nil ideal of $R$. Also by (2), there exist \( k \in \wedge \) and \( s \in S \) such that \( s I_n \subseteq I_k \) for every \( n \in \wedge \). Consequently, we obtain $sI = s \left(\bigcup\limits_{n\in\wedge} I_n \right) \subseteq I_k$. Since $\mathcal{D}$ is $S$-saturated, it follows that \( I \in \mathcal{D} \), as required. Applying Zorn's lemma, we conclude that $\mathcal{D}$ has a maximal element.\\
	
	\item (3) $\Rightarrow$ (4). Let $\mathcal{D}$ be a nonempty set of nil ideals of $R$. Consider the family $\mathcal{D}^{S}$ of all nil ideals $L \subseteq R$ such that there exist some $s\in S$ and $L_0 \in \mathcal{D}$ with $sL \subseteq L_0$. Clearly, $\mathcal{D}\subseteq \mathcal{D}^{S}$, so $\mathcal{D}^{S}\neq\emptyset$. It is straightforward to see that $\mathcal{D}^{S}$ is $S$-saturated. Thus, by (3) $\mathcal{D}^{S}$ has a maximal element $K\in\mathcal{D}^{S}$. Fix $s \in S$ and $Q\in \mathcal{D}$ such that $sK\subseteq Q$. Now, we claim that $Q$ is an $S$-maximal element of $\mathcal{D}$; specifically, given $L\in \mathcal{D}$ with $Q\subseteq L$, we will show that $sL \subseteq Q$. Note that $K + L$ satisfies $s(K + L)= sK + sL \subseteq Q + L \subseteq L$, so that $K + L \in \mathcal{D}^{S}$. Also,  $K+L$ is a nil ideal of $R$. Therefore maximality of $K$ implies $K = K + L$, so that $L \subseteq K$. But then $sL \subseteq sK \subseteq Q$, as desired.\\
	
	\item (4) $\Rightarrow$ (1). Let $I$ be a nil ideal of $R$. Let $\mathcal{D}$ denote  the family of finitely generated nil ideal $J$ of $R$ such that $J\subseteq I$, which is nonempty as $0\in\mathcal{D}$. Then $\mathcal{D}$ has an $S$-maximal element $K\in\mathcal{D}$. Fixing $x\in I$, take a finitely generated ideal of the form $Q= K + xR$. Since $K\subseteq I$ and $x\in I$, so $Q\subseteq I$. Consequently, $Q\in\mathcal{D}$ such that $K\subseteq Q$. This implies that there exists $s\in S$ such that $sQ\subseteq K$; in particular, $sx\in K$. This verifies $sI\subseteq K\subseteq I$, so that $I$ is $S$-finite. It follows that $R$ is Nil-$S$-Noetherian.
	
\end{proof}

Following from $\cite{ap20}$ a ring $R$ is said to be  \textit{$S$-reduced} if for any $r\in R$ such that $r^{n}=0$ for some $n \in\mathbb{N}$, then there exists $s\in S$ such that $sr= 0$. We introduce the notion of the uniform $S$-version of a reduced ring as follows:
 \begin{deff}\label{p1}
A ring $R$ is said to be \textit{uniformly-$S$-reduced} if there exists an element $s \in S$ such that for every $a \in R$, whenever $a^n = 0$ for some $n \in \mathbb{N}$, we have $sa = 0$.
\end{deff}
\begin{eg}\label{p2}
	Let $R=\mathbb{Z}_{24}$, and $S=\{\bar{1}, \bar{2}, \bar{4}, \bar{8}, \overline{16}\}$. Evidently, $Nil^*(R) = \{\bar{6}, \overline{12}, \overline{18} \}$ is the set of all non-zero nilpotent elements of $R$. Let $x\in Nil(R)$. Then we have $x^{3}=\bar{0}$. Take $s=\bar{4}\in S$. Then $sx=\bar{0}$, and hence $R$ is an uniformly-$S$-reduced ring with respect to $\bar{4}\in S$.
\end{eg}

In Definition~\ref{p1}, the element $s$ does not depend on the choice of $a \in R$; it depends only on the ring $R$ itself. It is clear from the definition, every uniformly-$S$-reduced ring is $S$-reduced ring. The following example shows that an $S$-reduced ring  need not be an uniformly-$S$-reduced ring.

	\begin{eg}
	\noindent
	Let $R=\dfrac{\mathbb{Q}[x_1,x_2,\ldots ,x_n,\ldots]}
	{\langle x_1,x_2^{2},x_3^{3},\ldots \rangle}$. Then $P= \operatorname{Nil}(R)=\dfrac{\langle x_1,x_2,\ldots ,x_n,\ldots \rangle}
	{\langle x_1,x_2^{2},x_3^{3},\ldots \rangle}$  is the unique prime ideal of $R$. Consider a multiplicative closed set $S=\{\bar{x_{i_1}}^{ n_{i_1}}\bar{x_{i_2}}^{n_{i_2}}\cdots \bar{x_{i_k}}^{n_{i_k}}\mid i_1,\ldots,i_k\in\mathbb{N}~~
	and~~n_{i_1},\ldots,n_{i_k}\in\mathbb{N}\cup\{0\}\}$. Since $\overline{x_i}^{\,i}=0$ in $R$ for all $i\in\mathbb{N}$, take $s_i=\overline{x_i}^{\,i-1}\in S$. Then $s_i\overline{x_i}=0$. Hence, $R$ is an $S$-reduced ring, but it is not a uniformaly
	$S$-reduced ring, because we cannot find a single $s\in S$
	such that $s\overline{x_i}=0$.
\end{eg}

\begin{prop}\label{n4}
Every uniformly-$S$-reduced ring is a Nil-$S$-Noetherian ring.
\end{prop}

\begin{proof}
 Let $R$ be a uniformly-$S$-reduced ring, and $I$ be a nil ideal of $R$. Let $\alpha\in I$. Then there exists $n\in\mathbb{N}$ such that $\alpha^{n}=0$. Since $R$ is uniformly-$S$-reduced ring, it follows that there exists fixed $t\in S$ such that $t\alpha=0$. This implies that $tI=(0)$ i.e., $I$ is $S$-finite. Therefore $R$ is a Nil-$S$-Noetherian ring. 
\end{proof}

\begin{prop}\label{n2}
Let $R$ be a Nil-$S$-Noetherian ring and $I$ be  a nil ideal of $R$. Then $R/I$ is a $\overline{S}$-Noetherian ring, where $\overline{S} = \{s+I: s\in S\}$.
\end{prop}
\begin{proof}
Let $K$ be a nil ideal of $R/I$. Then $K = J/I$ for some ideal $J$ of $R$ containing $I$. If $K\cap \overline{S}\neq\emptyset$, then $K$ is $\overline{S}$-finite, by \cite[Proposition 2(a)]{ad02}. Now, let $K\cap \overline{S}=\emptyset$ and $a+I\in K$, where $a\in J\setminus I$. Then there exists $n\in\mathbb{N}$ such that $(a+I)^{n}=I$ since $K$ is nil. It follows that $a^{n}\in I$. Since $I$ is nil, there exist $k\in\mathbb{N}$ such that $a^{nk}=0$. Put $m=nk\in\mathbb{N}$. Thus $a^{m}=0$, and hence $J$ is a nil ideal. Since $R$ is Nil-$S$-Noetherian, there exist $s\in S$ and $x_{1},\ldots, x_{n}\in J$ such that $sJ\subseteq \langle x_1,\ldots, x_{n} \rangle\subseteq J$. Let $x\in J$. Then we can find $a_1,\ldots, a_n\in R$ such that $sx=a_1x_1+\cdots+a_nx_n$. It follows that $(s+I)(x+I)=(a_1+I)(x_1+I)+\cdots+(a_n+I)(x_n+I)$, where $s+I\in\overline{S}$ and $a_{1}+I,\ldots, a_{n}+I\in R/I$. This implies that $(s+I)(J/I)\subseteq \langle x_1+I,\ldots, x_n+I \rangle\subseteq J/I$, i.e., $K$ is $\overline{S}$-finite. Hence $R/I$ is an $\overline{S}$-Noetherian ring. 
\end{proof}

\begin{prop}
Let $R = \prod\limits_{i=1}^{n} R_{i}$ be a finite direct product of rings and 
let $S \subseteq R$ be a multiplicatively closed set. For each $i = 1, \dots, n$, 
let $S_{i} = \pi_{i}(S)$, which is a multiplicatively closed subset of $R_{i}$, 
where $\pi_{i}: R \to R_{i}$ denotes the natural projection. Then 
$R$ is a Nil-$S$-Noetherian ring if and only if each $R_{i}$ is a Nil-$S_{i}$-Noetherian ring.

\end{prop}

\begin{proof}
Suppose $R$ is Nil-$S$-Noetherian and let $I_{i}$ be a nil ideal of $R_{i}$. Then the corresponding nil ideal of $R$ is $E_{i}=\{0\}\times \cdots \times I_{i}\times \cdots \times \{0\}$. The projection of $E_{i}$ onto the $i$-th coordinate is precisely $I_{i}$. 
Since $R$ is Nil-$S$-Noetherian, $E_{i}$ must be $S$-finite. Hence, there exist $s\in S$ and a finitely generated ideal $J\subseteq R$ such that $sE_{i}\subseteq J \subseteq E_{i}$. Applying the projection map $\pi_{i}:R\to R_{i}$, we obtain $\pi_{i}(s)I_{i}=\pi_{i}(sE_{i}) \subseteq \pi_{i}(J)\subseteq I_{i}$, where $\pi_{i}(s)\in S_{i}=\pi_{i}(S)$ and $\pi_{i}(J)$ is a finitely generated ideal of $R_{i}$. Consequently, $I_{i}$ is $S_{i}$-finite, and therefore 
$R_{i}$ is Nil-$S_{i}$-Noetherian.

Conversely, suppose each $R_{i}$ is Nil-$S_{i}$-Noetherian. Let $I$ be a nil ideal of $R$. Then $I = I_{1} \times \cdots \times I_{n}$, where each $I_{i}$ is a nil ideal of $R_{i}$. By assumption, each $I_{i}$ is $S$-finite, and hence $I$ is $S$-finite in $R$. Therefore, $R$ is Nil-$S$-Noetherian.
\end{proof}

\begin{prop}
Let $\varphi:R\to R'$ be a ring homomorphism making $R$ a module
retract of $R'$, i.e., there exists an $R$-module homomorphism $\rho:R'\to R$ with $\rho\circ\varphi=\operatorname{id}_R$. If $R'$ is Nil-$\varphi(S)$-Noetherian, where $\varphi(S)$ is a multiplicatively closed subset of $R'$ with $0\notin \varphi(S)$, then $R$ is Nil-$S$-Noetherian.
\end{prop}

\begin{proof}
Let $I$ be a nil ideal of $R$. For $i\in I$, since $\varphi$ is a ring homomorphism, we have $\varphi(i)^n=\varphi(i^n)=\varphi(0)=0$,
so $\varphi(i)$ is nilpotent in $R'$. Hence every element of $\varphi(I)$ is nilpotent. Now consider the ideal $\varphi(I)R'=\Big\{\;\sum_{j}\varphi(i_j)r'_j \;\Big|\; i_j\in I,\; r'_j\in R'\;\Big\}$.
This is the ideal generated by the nilpotent set $\varphi(I)$, and it is well known that an ideal generated by nilpotent elements is a nil ideal. Thus $\varphi(I)R'$ is a nil ideal of $R'$. Since $R'$ is Nil-$\varphi(S)$-Noetherian, there exist $s'\in \varphi(S)$ and a finitely generated ideal $J'\subseteq R'$ such that $s'\varphi(I)R'\subseteq J' \subseteq \varphi(I)R'$. Choose $s\in S$ with $\varphi(s)=s'$. Applying $\rho$ to the chain gives $\rho(s'\varphi(I)R') \subseteq \rho(J') \subseteq \rho(\varphi(I)R')$. We claim that $\rho(s'\varphi(I)R')=sI$.  
Indeed, if $x'\in s'\varphi(I)R'$, then $x'=s'\sum_j\varphi(i_j)r'_j$, so $\rho(x')=\sum_j \rho(s')\,\rho(\varphi(i_j))\,\rho(r'_j)
= \sum_j s i_j \rho(r'_j)\in sI$. Hence $\rho(s'\varphi(I)R')\subseteq sI$. Conversely, every element of $sI$ has the form $\sum_k (s i_k) r_k$ with $i_k\in I,\; r_k\in R$, and
$\sum_k (s i_k) r_k = \rho\!\left(\sum_k \varphi(s)\varphi(i_k)\varphi(r_k)\right)\in \rho(s'\varphi(I)R')$. Thus $sI\subseteq \rho(s'\varphi(I)R')$, proving the equality. Finally, note that $\rho(\varphi(I)R')\subseteq I$. Therefore, $sI=\rho(s'\varphi(I)R') \subseteq \rho(J') \subseteq \rho(\varphi(I)R')\subseteq I$. Since $J'$ is finitely generated ideal in $R'$, its image $\rho(J')$ is finitely generated in ideal in $R$. Thus there exists $s\in S$ with $sI\subseteq \rho(J')\subseteq I$, showing that $I$ is $S$-finite. Hence $R$ is Nil-$S$-Noetherian.

\end{proof}

\begin{theorem} \label{cohen}\textbf{(Cohen theorem for Nil-$S$-Noetherian rings).}
Let $R$ be a ring. Then the following statements are equivalent:
\begin{enumerate}
    \item $R$ is Nil-$S$-Noetherian. 
    \item Every prime ideal $P \subseteq Nil(R)$ with $P \cap S = \emptyset$ is $S$-finite.
\end{enumerate}
\end{theorem}

\begin{proof}
$(1) \Rightarrow (2)$:  The result is immediate from the definition \ref{n1}.

\noindent
$(2) \Rightarrow (1)$:  
Suppose every prime ideal $P \subseteq Nil(R)$ with $P \cap S = \emptyset$ is $S$-finite. Let $I$ be a nil ideal of $R$. On the contrary, assume that $I$ is not $S$-finite. Consider the collection $\mathcal{F} = \{J \subseteq Nil(R) \mid J~is~ ~not~$S$-finite \}$. Clearly $\mathcal{F}$ is nonempty since $I \in \mathcal{F}$. Ordered by inclusion, every chain in $\mathcal{F}$ has an upper bound given by its union, which is again a nil, non-$S$-finite ideal. By Zorn's lemma, there exists a maximal element $P\in \mathcal{F}$. We claim that $P$ is a prime ideal. Clearly, $P\cap S=\emptyset$ otherwise $P$ is automatic $S$-finite, by \cite[Proposition 2(a)]{ad02}. Suppose there exist $a, b \in R \setminus P$ such that $ab \in P$. Then $P \subsetneq (P+Ra)\cap Nil(R)$. By maximality of $P$, the nil ideal $(P+Ra)\cap Nil(R)$ must be $S$-finite. Hence there exist $s \in S$, $\alpha_{1},\ldots, \alpha_{n} \in (P+Ra)\cap Nil(R)$ such that $s\big((P+Ra)\cap Nil(R)\big) \subseteq \langle \alpha_{1},\ldots,\alpha_{n} \rangle \subseteq (P+Ra)\cap Nil(R)$. Now consider the ideal $Q=(P:a)=\{x\in R\mid ax\in P\}$. Evidently, $b \in Q \setminus P$, so $P \subsetneq Q$. Define $Q' = Q \cap Nil(R)$, which is a nil ideal strictly containing $P$. By maximality of $P$, $Q'$ is $S$-finite. Thus there exist $t\in S$ and $\beta_{1},\ldots,\beta_{k}\in Q'$ such that $tQ' \subseteq \langle \beta_{1},\ldots,\beta_{k}\rangle \subseteq Q'$. Let $x \in P$. Since $sx\in s\big((P+Ra)\cap Nil(R)\big)$, we can write $sx = u_{1}\alpha_{1} + \cdots + u_{n}\alpha_{n}$ with $u_{i}\in R$. Each $\alpha_{i}$ has the form $p_{i} + ax_{i}$ with $p_{i}\in P$, $x_{i}\in R$ and $ax_{i}\in Nil(R)$. Hence $sx = u_{1}p_{1} + \cdots + u_{n}p_{n} + a(u_{1}x_{1} + \cdots + u_{n}x_{n})$. Thus $a(u_{1}x_{1}+\cdots+u_{n}x_{n}) \in P\subseteq Nil(R)$, which implies $u_{1}x_{1}+\cdots+u_{n}x_{n} \in (P:a)=Q$.  Therefore $u_{1}ax_{1}+\cdots+u_{n}ax_{n}\in Q\cap Nil(R)=Q'$. Since $Q'$ is $S$-finite, there exist $w_{1},\ldots,w_{k}\in R$ such that $t(au_{1}x_{1}+\cdots+au_{n}x_{n}) = w_{1}\beta_{1} + \cdots + w_{k}\beta_{k}$. Now, we have $stx = t(u_{1}p_{1}+\cdots+u_{n}p_{n}) + w_{1}\beta_{1}+\cdots+w_{k}\beta_{k}$. This shows that
$uP \subseteq \langle p_{1},\ldots,p_{n},\beta_{1},\ldots, \beta_{k} \rangle \subseteq P$, where $u=st\in S$. Hence $P$ is $S$-finite, contradicting the choice of $P \in \mathcal{F}$. Therefore no such $P$ exists, and so every nil ideal is $S$-finite. Thus $R$ is Nil-$S$-Noetherian.
\end{proof}

 In noncommutative ring theory, one of the classical results is 
\textbf{Levitzki's theorem}, which asserts that every nil one-sided 
ideal of a right Noetherian ring is necessarily nilpotent. Motivated by the importance of the $S$-version, Baeck et al. \cite{jw16} demonstrated that a nil one-sided ideal of a right $S$-Noetherian ring need not be nilpotent (see ~\cite[Example 2.5]{jw16}). We extend this observation to the commutative setting by proving that the same phenomenon occurs for commutative $S$-Noetherian rings when $S \cap Z(R) = \emptyset$, where $Z(R)$ denotes the set of all zero divisors of $R$. To illustrate this, we first present an example of a commutative non-Noetherian ring that admits a nil ideal which is not nilpotent.

\begin{eg}
Let $K$ be a field and consider the non-Noetherian  ring $R = \dfrac{K[x_1, x_2, x_3, \dots]}{\langle x_1^2, x_2^2, x_3^2, \dots \rangle}$. In $R$, each variable satisfies $x_i^2 = 0$, so every $x_i$ is nilpotent. 
Let $I = \langle x_1, x_2, x_3, \dots \rangle \subseteq R$. Then $I$ is a nil ideal since every element of $I$ is nilpotent. However, $I$ is not nilpotent: for each $n \geq 1$ we have $x_1 x_2 \cdots x_n \in I^n \quad \text{and} \quad x_1 x_2 \cdots x_n \neq 0 \text{ in } R$. Thus $I^n \neq 0$ for all $n$, showing that $I$ is a nil ideal but not nilpotent. 

\end{eg}
To establish Levitzki's theorem in the context of $S$-Noetherian rings and Nil-$S$-Noetherian rings, we first require some preliminary 
lemmas, which we state below.

\begin{lem}\label{l0}
Let $R$ be a ring, $S \subseteq R$ be a multiplicatively closed subset with $S \cap Z(R) = \emptyset$. If $I$ and $J$ are ideals of $R$, then there exists $s \in S$ such that $s (I \cap J)= sI \cap sJ$.
\end{lem}

\begin{proof}
    Let $\alpha\in s(I\cap J)$. Then $\alpha=s\beta$, where $\beta\in I\cap J$. This implies that $s\beta\in sI$ and $s\beta\in sJ$, $\alpha=s\beta\in sI\cap sJ$. Thus $s(I\cap J)\subseteq sI\cap sJ$. On the other hand, let $\alpha\in sI\cap sJ$. Then $\alpha=s\beta_1=s\beta_2$ for some $\beta_1\in I$, $\beta_2\in J$. Consequently, $s(\beta_1-\beta_2)=0$, $\beta_1=\beta_2$ since $S\cap Z(R)=\emptyset$. Therefore $\alpha\in s(I\cap J)$, and hence $s(I\cap J)=sI\cap sJ$.
\end{proof}
\noindent
In Lemma~\ref{l0}, the assumption $S \cap Z(R) = \emptyset$ is essential. We illustrate this by providing a counterexample.
\begin{eg}
\noindent
 Let \(R=\mathbb{Z}_{4}\times\mathbb{Z}_{2}\) and $S= \{(0,0), (1,1), (2,0)\}$. Consider the ideals $I=\langle(1,0)\rangle$, and  $J=\langle(1,1)\rangle$, and take \(s=(2,0)\in S\cap Z(R)\). Then we have 	$s(I\cap J)= \langle(0,0)\rangle$ and $sI\cap sJ = \langle(2,0)\rangle$. Hence $  s(I\cap J) \neq sI\cap sJ$, as desired.
\end{eg}

\begin{lem}\label{l1}
		Let $R$ be an $S$-Noetherian ring and $I$ be an ideal disjoint with $S$. Then there exist $s\in S$ and an integer $m$ such that $s(rad(I))^m\subseteq I$, where $rad(I)=\{x\in R| \ x^n\in I \text{ for some } n\in \mathbb{N}\}$ is the radical ideal of $I$.
	\end{lem}
	\begin{proof}
		Since $rad(I)$ is  $S$-finite, $t(rad(I))\subseteq J\subseteq rad(I)$ for some $t\in S$ and a finitely generated ideal $J=\langle x_1,x_2,\dots,x_l \rangle$ of $R$. Suppose $n_i\in \mathbb{N}$ be such that $x_i^{n_i}\in I$. Take $m=\sum_{i=1}^l (n_i-1)+1$. Then $J^m\subseteq I$. Consequently, $s(rad(I))^m\subseteq J^m\subseteq I$, where $s=t^m$. 
	\end{proof}
    
\begin{theorem} \label{Levitzki-1}\textbf{($S$-version of Levitzki's theorem for $S$-Noetherian Rings)}
Let $R$ be a ring, $S\subseteq R$ be a multiplicative set with $S\cap Z(R) = \emptyset$. If $R$ is an $S$-Noetherian ring, then every nil ideal $I$ with $I\cap S=\emptyset$ is nilpotent.
\end{theorem}

\begin{proof}
Since $R$ is $S$-Noetherian, it follows from \cite[Theorem~2.9]{ts24} that the zero ideal $(0)$ admits a minimal $S$-primary decomposition $(0) = \bigcap_{i=1}^{n} Q_{i}$, where each $Q_i$ is a $P_i$-$S$-primary ideal and $P_{i} = rad(Q_{i})$ is an $S$-prime ideal. Consequently, $Nil(R)= rad(0)=\bigcap\limits_{i=1}^{n} P_{i}$.
By Lemma~\ref{l1}, for each $i=1,\ldots,n$, there exist an integer $k_i \geq 1$ and an element $s_i \in S$ such that $s_{i} P_i^{\,k_i}\subseteq Q_i$. Let $k = \max\{k_1, \dots, k_n\}$ and $s = \prod_{i=1}^{n} s_{i}$. Then, by Lemma~\ref{l0}, $s \,(Nil(R))^{k} 
= s \,\Big(\bigcap\limits_{i=1}^{n} P_{i}\Big)^{k}\subseteq\; s \Big(\bigcap\limits_{i=1}^{n} P_{i}^{\,k_{i}}\Big)
= \bigcap\limits_{i=1}^{n} sP^{k_{i}}_{i}$. Consequently, $s (Nil(R))^{k}= sP^{k_{1}}_{1}\cap \cdots \;\cap\;sP^{k_{n}}_{n}\subseteq s_1P^{k_{1}}_{1}\cap \cdots \;\cap\; s_nP^{k_{n}}_{n}\subseteq Q_{1}\cap\cdots\cap Q_{n}
= (0)$, by Lemma \ref{l1}. Hence the nilradical $Nil(R)$ is $S$-nilpotent. Since $I \subseteq Nil(R)$, it follows that $sI^k \;\subseteq\; s(Nil(R))^k \;=\; (0)$, $sI^{k}=0$. Since $S\cap Z(R)=\emptyset$, so $I^{k}=0$. Thus $I$ is a nilpotent ideal of $R$.
\end{proof}
Next, we extend the above result to the setting of Nil-$S$-Noetherian 
rings, where no additional condition on $S$ is required.

\begin{theorem} \label{Levitzki-2}\textbf{($S$-version of Levitzki's theorem for Nil-$S$-Noetherian Rings)}
Let $R$ be a commutative ring and let $S\subseteq R$ be a multiplicative subset. If $R$ is Nil-$S$-Noetherian, then every nil ideal $I$ (disjoint from $S$) is  $S$-nilpotent.
\end{theorem}

\begin{proof}
Let $I$ be a nil ideal of $R$. By the Nil-$S$-Noetherian hypothesis, $I$ is $S$-finite; hence there exist $s\in S$ and a finitely generated ideal $J\subseteq R$ with $sI \subseteq J\subseteq I$. Write $J=\langle a_1,\dots,a_m \rangle$, where $a_{i}\in R$ for all $i=1,\ldots, m$. Since $J\subseteq I$ and $I$ is nil, each generator $a_i$ is nilpotent. Thus for each $i$ there exists $n_i\ge1$ with $a_i^{n_i}=0$. Set $N=\sum\limits_{i=1}^m (n_i-1)+ 1$. Thus $J^N = 0$. Now observe that $(sI)^N = s^N I^N \subseteq J^N = 0$, so $s^N I^N = 0$. Setting $t= s^N \in S$ and $n= N$ yields $tI^n = 0$, which shows that $I$ is $S$-nilpotent.
\end{proof}

Recall \cite{ad02} that a multiplicatively closed subset $S$ of a ring $R$ is said to be \textit{anti-Archimedean} if for each $s \in S$, $\left(\bigcap_{n=1}^{\infty} s^n R \right) \cap S \neq \emptyset$.

\begin{theorem}\label{Hilbert-1} \textbf{(Hilbert basis theorem for Nil-$S$-Noetherian rings)} Let $R$ be a ring and $S \subseteq R$ an anti-Archimedean multiplicatively closed subset. If $R$ is Nil-$S$-Noetherian, then $R[x]$ is Nil-$S$-Noetherian.
\end{theorem}
\begin{proof}
Let $P$ be a nil prime ideal of $R[x]$ such that $P\cap S\neq\emptyset$. For each $n\ge0$, define $P_n=\{\,a\in R\mid \exists~  f(x)\in P \text{ with } \deg f=n \text{ and leading coefficient } a\,\}\cup\{0\}$. Since a polynomial is nilpotent if and only if each of its coefficients is nilpotent, it follows that each $P_n$ is a nil ideal of $R$. Suppose $r$ is the smallest integer such that $P_r\neq0$. Then $P_r\subseteq P_{r+1}\subseteq\cdots$ is an ascending chain of ideals of $R$. Also $P\cap S\neq\emptyset$ and $0\notin S$, so $P_i\cap S=\emptyset$ since $P_i\subseteq Nil(R)$ and $Nil(R)\cap S=\emptyset$. Let $I=\bigcup_{n\ge r}P_n$. Then $I$ is a nil ideal of $R$ as each $P_n$ is nil. Since $R$ is nil-$S$-Noetherian, $I$ is $S$-finite. Hence there exists $t\in S$ such that $tI\subseteq\langle b_1,b_2,\ldots,b_m\rangle$ for some $b_i\in R$. Take $f_i\in P$ such that $f_i(x)=b_i x^{d_i}+\text{lower terms}$ for $i=1,2,\ldots,m$. Put $d=\max\{d_i\}_{i=1}^{m}$. Since each $P_k$ is a nil ideal of $R$ and $R$ is nil-$S$-Noetherian, there exists $t_k\in S$ such that $t_kP_k\subseteq\langle a_{k1},a_{k2},\ldots,a_{kn_k}\rangle\subseteq P_k$. For each $a_{ki}$, take $f_{ki}\in P$ such that $f_{ki}(x)=a_{ki}x^k$ for all $k\in\{r,r+1,\ldots,d\}$. Let $f=\sum_{i=0}^{\nu}p_ix^i\in P$. If $\nu>d$, then $p_\nu\in I$. This implies that $tp_\nu=r_1b_1+r_2b_2+\cdots+r_mb_m$ for some $r_i\in R$. Hence $tf-tp_\nu x^\nu=tf-\sum_{i=1}^{m}r_ib_ix^\nu=tf-\sum_{i=1}^{m}r_if_i x^{\nu-d_i}\in P$ is a polynomial of degree at most $\nu-1$. By repeating this process, we obtain a polynomial $h=\sum_{i=1}^{l}q_ix^i$ with $l\le d$. Since $q_l\in P_l$, we have $t_lq_l\in\langle a_{l1},a_{l2},\ldots,a_{ln_l}\rangle\subseteq P_l$. Write $t_lq_l=c_1a_{l1}+c_2a_{l2}+\cdots+c_{n_l}a_{ln_l}$ for some $c_{i} \in R$. Then $t_lh-t_lq_lx^l=t_lh-\sum_{i=1}^{n_l}c_ia_{li}x^l=t_lh-\sum_{i=1}^{n_l}c_if_{li}$ is a polynomial of degree at most $l-1$. Repeating this process, we obtain $t^{\nu-l}\prod_{k=r}^{l}t_k f=g_1+g_2$ for some $g_1\in\langle f_1,\ldots,f_m,f_{r1},\ldots,f_{rn_r},f_{(r+1)1},\ldots,f_{(r+1)n_{r+1}},\ldots,f_{d1},\ldots,f_{dn_d}\rangle$ and $g_2\in P$ with degree at most $r-1$. But $r$ was chosen to be the least integer such that $P_r\neq0$, hence $g_2=0$. Since $S$ is anti-Archimedean, $(\bigcap_{n\ge1}t^nR)\cap S\neq\emptyset$. Let $w\in(\bigcap_{n\ge1}t^nR)\cap S$. Then $w\prod_{k=r}^{l}t_k f\in\langle f_1,\ldots,f_m,f_{r1},\ldots,f_{rn_r},f_{(r+1)1},\ldots,f_{(r+1)n_{r+1}},\ldots,f_{d1},\ldots,f_{dn_d}\rangle$.  Put $s=w\prod_{k=r}^{l}t_k$. Then $sP\subseteq\langle f_1,\ldots,f_m,f_{r1},\ldots,f_{rn_r},f_{(r+1)1},\ldots,f_{(r+1)n_{r+1}},\ldots,f_{d1},\ldots,f_{dn_d}\rangle\subseteq P$. Therefore $P$ is $S$-finite. Hence by Theorem~\ref{cohen}, $R[x]$ is a Nil-$S$-Noetherian ring.

\end{proof}

\begin{theorem} \label{Hilbert-2}
	Let $R$ be a ring and $S \subseteq R$ be an anti-Archimedean multiplicative set  consisting of non-zero divisors. If $R$ is Nil-$S$-Noetherian, then the power series ring $R[[x]]$ is also Nil-$S$-Noetherian.
\end{theorem}

\begin{proof}
Let $P$ be a nil prime ideal of $R[[x]]$ such that $P\cap S\neq\emptyset$.	 Let
	$ P(0) = \{f(0) \mid f(x) \in P\} $. 
	Let $f= \sum_{i=0}^{\infty}a_{i}x^{i} \in P$. Then $f$ is nilpotent as $P$ is nil ideal. This implies that each coefficient $a_{n}$ is nilpotent, in particular, its constant term $a_{0} = f(0)$ is nilpotent.
	Consequently, $P(0)$ is a nil ideal of $R$. This implies that $P(0)$ is $S$-finite since $R$ is Nil-$S$-Noetherian. Thus, there exists $s \in S$ and a finite set of polynomials $\{p_1(x), p_2(x), \dots, p_n(x)\} \subseteq P$ such that
	$ sP(0) \subseteq \langle p_1(0), p_2(0), \dots, p_n(0) \rangle \subseteq P(0) $.
	If $x \in P$, then we can write $P = P(0) + xR[[x]]$. It follows that
	 $sP \subseteq \langle p_1(0), p_2(0), \dots, p_n(0), x \rangle \subseteq P$.
	Hence $P$ is $S$-finite.
	 Now suppose  $x \notin P$.
	Let $f \in P$. Then $f(0) \in P(0)$. Write
	$ sf(0) = \sum_{i=1}^{n} r_{0i} p_i(0) \quad \text{for some } r_{0i} \in R$.
	This implies that
	$ sf -  sf(0) = sf- \sum_{i=1}^{n} r_{0i} p_i(0) = xf_1 $
	for some $f_1 \in R[[x]]$. Since $xf_1 \in P$ and $x \notin P$ (as $P$ is a prime ideal), it must be that $f_1 \in P$. 	
	By repeating this iterative process, we can similarly write
	$sf_1 - \sum_{i=1}^{n} r_{1i} p_i(0) = xf_2 $,
	where $f_2 \in P$. Continuing this construction, we obtain the representation
	$ f = \sum_{i=1}^{n} p_i(0) \left( \sum_{j=0}^{\infty} \frac{r_{ji}}{s^{j+1}} x^j \right) $.
	Note that the coefficients $\frac{r_{ji}}{s^{j+1}}$ belong to the localization $S^{-1}R$. Since $S$ is given to be an anti-Archimedean set, we have
	$ \left( \bigcap_{i=1}^{\infty} s^i R \right) \cap S \neq \emptyset $.
	Let $t \in \left( \bigcap_{i=1}^{\infty} s^i R \right) \cap S$. Then for each $j$, $t = s^j c_j$ for some $c_j \in R$. 
	Multiplying the expression for $f$ by $t$, we get
	$ tf = \sum_{i=1}^{n} p_i(0) \left( \sum_{j=0}^{\infty} c_{j+1} r_{ji} x^j \right) \in \langle p_1(0), \dots, p_n(0) \rangle $.
	This implies that
	$ tP \subseteq \langle p_1(0), p_2(0), \dots, p_n(0) \rangle \subseteq P$.
	Thus $P$ is $S$-finite. Hence by Theorem~\ref{cohen}, $R[[x]]$ is a Nil-$S$-Noetherian ring.
	
\end{proof}

It is well known that the Lasker-Noether theorem holds for Noetherian rings; that is, every ideal of a Noetherian ring can be expressed as a finite intersection of primary ideals. Singh et al.~\cite{ts24} established the existence and uniqueness of $S$-primary decomposition in $S$-Noetherian modules. Finally, we prove the existence of $S$-primary decomposition for Nil-$S$-Noetherian rings.

\begin{deff}\label{S-iir.}\cite{ts24}
	An ideal $Q$ (disjoint from $S$) of the ring $R$ is called  $S$-irreducible if $s(I\cap K)\subseteq Q \subseteq I\cap K$ for some $s\in S$ and some ideals $I$, $K$ of $R$, then there exists $s'\in S$ such that either $ss'I\subseteq Q$ or $ss'K\subseteq Q$.
\end{deff}

\noindent
It is clear from the definition that every irreducible ideal is an $S$-irreducible ideal. However, the following example shows that an $S$-irreducibile ideal need not be irreducible. 
\begin{eg}\label{fm}
	\noindent
	Let $R=\mathbb{Z}$, $S=\mathbb{Z}\setminus 2\mathbb{Z}$ and $I=12\mathbb{Z}$. Since $I=3\mathbb{Z}\cap 4\mathbb{Z}$, therefore $I$ is not an irreducible ideal of $R$. Now, take $s=3\in S$. Then $3(4\mathbb{Z})=12\mathbb{Z}\subseteq I$. Thus $I$ is an $S$-irreducible ideal of $R$.
\end{eg}
\noindent
Recall \cite[Definition 2.1]{me22}, a proper ideal $Q$ of a ring $R$ disjoint from $S$ is said to be $S$-primary if there exists an $s \in S$ such that for all $a, b \in R$, if $ab \in Q$, then either $sa \in Q$ or $sb \in rad(Q)$. Following from \cite{ts24}, let
$I$ be an ideal of $R$ such that $I\cap S=\emptyset$. Then $I$ admits $S$-primary decomposition if $I$ can be written as a finite intersection of $S$-primary ideals of $R$.\\

\noindent
Now, we extend Lasker-Noether theorem for Nil-$S$-Noetherian rings. We start with the following lemma.
\begin{lem}\label{sp}
	\noindent
	Let $R$ be a Nil-$S$-Noetherian ring. Then every $S$-irreducible nil ideal of $R$ is an $S$-primary.
\end{lem}
\begin{proof}
	Suppose $Q$ is an $S$-irreducible nil ideal of $R$. Let $a,b\in R$ be such that $ab\in Q$ and $sb\notin Q$ for all $s\in S$. Our aim is to show that there exists $t\in S$ such that $ta\in rad(Q)$. Consider $A_n=\{x\in R\hspace{0.2cm}|\hspace{0.1cm} a^{n}x\in Q \}$ for $n\in\mathbb{N}$. Let $y\in A_{n}$. Then $a^{n}y\in Q$. Since $Q$ is nil, there exists $k\in\mathbb{N}$ such that $(a^{n}y)^{k}=0$. Write $y^{k}=y^{k} a^{nk-nk}=(a^{n}y)^{k} a^{-nk}=0$. This implies that each $A_{n}$ is a nil ideal for all $n\in\mathbb{N}$, and $A_1\subseteq A_2\subseteq A_3\subseteq \cdots$ is an increasing chain of nil ideals of $R$. Since $R$ is a Nil-$S$-Noetherian, by Theorem \ref{j1}, this chain is $S$-stationary, i.e., there exist $k\in \mathbb{N}$ and $s\in S$ such that $sA_n\subseteq A_k$ for all $n\geq k$. Consider the two ideals $I=(a^{k}) +\hspace{0.1cm} Q$ and $K=(b) +\hspace{0.1cm} Q$ of $R$. Then  $Q\subseteq I\cap K$. For the reverse containment, let $u\in I\cap K$. Write  $u=a^{k}v+q$ for some $v\in R$ and $q\in Q$. Since $ab\in Q$, $aK\subseteq Q$; whence $au\in Q$. Now, $a^{k+1}v=a(a^{k}v)=a(u-q)\in Q$. This implies that $v\in A_{k+1}$, and so  $sv\in sA_{k+1}\subseteq A_k$. Consequently, $a^{k}sv\in Q $ which implies that  $a^{k}sv +sq=su\in Q$. Thus we have $s(I\cap K)\subseteq Q\subseteq I\cap K$. This implies that there exists $s'\in S$ such that either $ss'I\subseteq Q$ or $ss'K\subseteq Q$ since $Q$ is $S$-irreducible. If $ss'K\subseteq Q$, then $ss'b\in Q$  which is not possible. Therefore $ss'I\subseteq Q$ which implies that $ss'a^{k}\in Q$. Put $t=ss'\in S$. Then $(ta)^{k} \in Q$ and hence  $ta\in rad(Q)$, as desired.
\end{proof}

\begin{theorem}\label{on}
	\noindent
	Let $R$ be a Nil-$S$-Noetherian . Then every proper nil ideal of $R$ disjoint with $S$ can be written as a finite intersection of $S$-primary ideals.
\end{theorem}

\begin{proof}
	\noindent
	Let $E$ be the collection of nil ideals of $R$ which are disjoint with $S$ and can not be written as a finite intersection of $S$-primary ideals. We wish to show $E=\emptyset$. On the contrary suppose  $E\neq\emptyset$.  Since $R$ is Nil-$S$-Noetherian ring, by Theorem \ref{j1}, there exists an $S$-maximal element in $E$, say $I$. Evidently, $I$ is not an $S$-primary ideal, by Lemma \ref{sp}, $I$ is not an $S$-irreducible ideal, and so  $I$ is not an irreducible ideal. This implies that  $I=K\cap L$ for some ideals $K$ and $L$ of $R$ with $I\neq K$ and $I\neq L$. As $I$ is not $S$-irreducible, and so $sK\nsubseteq I$ and $sL\nsubseteq I$ for all $s\in S$. Now, we claim that $K$, $L\notin E$. For this, if $K$ (respectively, $L$) belongs to $E$, then since $I$ is an $S$-maximal element of $E$ and $I\subset K$ (respectively, $I\subset L$), there exists $s'$ (respectively, $s''$) from $S$ such that $s'K\subseteq I$ (respectively, $s''L\subseteq I$). This is not possible, as $I$ is not $S$-irreducible. Therefore $K, L\notin E$. Also, if $K\cap S\neq\emptyset$ $(respectively, L\cap S\neq\emptyset)$, then there exist $s_{1}\in K\cap S$  $(respectively, s_{2}\in L\cap S)$. This implies that $s's_{1}\in s'K\subseteq I$ $(respectively, s''s_{2}\in s''L\subseteq I)$, which is a contradiction because $I$ disjoint with $S$. Thus $K$ and $L$ are also disjoint with $S$. This implies that $K$ and $L$ can be written as a finite intersection of $S$-primary ideals. Consequently, $I$ can also be written as a finite intersection of $S$-primary ideals since  $I=K\cap L$, a contradiction as $I\in E$. Thus $E=\emptyset$, i.e., every proper nil ideal of $R$ disjoint with $S$ can be written as a finite intersection of $S$-primary ideals. 
\end{proof}
\noindent
\textbf{Conflict of Interest Statement:} The authors declare that they have no conflict of interest.\\
\textbf{Funding:} The authors did not receive any specific funding for this work.

\end{document}